\documentclass[12pt]{amsart}
\usepackage[latin1]{inputenc}
\usepackage[english]{babel}
\usepackage{amsmath,amssymb,amsthm,amsfonts, color}
\usepackage{latexsym,dsfont}
\usepackage{enumerate}
\usepackage{graphicx}
\usepackage{latexsym}

\makeatletter
\@namedef{subjclassname@2010}{ \textup{2010} Mathematics Subject Classification}
\makeatother

\newtheorem{theorem}{Theorem}%[section]

\newtheorem{lemma}{Lemma}%[section]
\newtheorem{corollary}[lemma]{Corollary}
\newtheorem{proposition}[lemma]{Proposition}

\theoremstyle{definition}

\newtheorem{example}[lemma]{Example}
\newtheorem{remark}[lemma]{Remark}

\numberwithin{equation}{section}

\numberwithin{equation}{section}

\begin{document}
\title{Abstract Ces\`aro spaces. I. Duality{\rm *}}\thanks{{\rm *}This publication has been produced during scholarship 
period of the first author at the Lule{\aa} University of Technology, thanks to a Swedish Institute scholarschip 
(number 0095/2013).}

\author[Le\'snik]{Karol Le\'snik}
\address[Karol Le{\'s}nik]{Institute of Mathematics\\
of Electric Faculty Pozna\'n University of Technology, ul. Piotrowo 3a, 60-965 Pozna{\'n}, Poland}
\email{\texttt{klesnik@vp.pl}}
\author[Maligranda]{Lech Maligranda}
\address[Lech Maligranda]{Department of Engineering Sciences and Mathematics\\
Lule{\aa} University of Technology, SE-971 87 Lule{\aa}, Sweden}
\email{\texttt{lech.maligranda@ltu.se}}
\maketitle

\vspace{-9mm}

\begin{abstract}
We study abstract Ces\`aro spaces $CX$, which may be regarded as generalizations of Ces\`aro sequence spaces $ces_p$ and Ces\`aro function spaces $Ces_p(I)$ on $I = [0,1]$ or $I = [0,\infty)$, and also as the description of optimal domain from which Ces\`aro operator acts to $X$. We find the dual of such spaces in a very general situation. What is however even more important, we do it in the simplest possible way. Our proofs are more elementary than the known ones for $ces_p$ and $Ces_p(I)$. This is the point how our paper should be seen, i.e. not as generalization of known results, but rather like grasping and exhibiting the general nature of the problem, which is not so easy visible in the previous publications. Our results show also an interesting phenomenon that there is a big difference between duality in the cases of finite and infinite interval.
\end{abstract}


\vspace{-2mm}

%%%%%%%%%%%%%%%%%%% Introduction
\section{Introduction}
In 1968 the Dutch Mathematical Society posted a problem to find the K\"othe dual of Ces\`aro sequence spaces $ces_p$ and Ces\`aro function spaces $Ces_p[0, \infty)$. In 1974 problem was solved (isometrically) by Jagers \cite{Ja74} even for weighted Ces\`aro 
sequence spaces, but the proof is far from being easy and elementary. Before, it was also known a result of Luxemburg and 
Zaanen \cite{LZ66} who have found the K\"othe dual of $Ces_{\infty}[0,1]$ space (known as the Korenblum-Kre{\u \i}n-Levin 
space -- cf. \cite{KKL48}). Already in 1957 Alexiewicz, in his ovelooked paper \cite{Al57}, found implicitly the K\"othe dual of the weighted 
$ces_{\infty}$-spaces (see Section 5 for more information). Later on some amount of papers appeared in the case of sequence spaces 
as well as for function spaces. Bennett \cite{Be96} proved representation of the dual $(ces_p)^*$ for $1 < p < \infty$ as the corollary 
from factorization theorems for Ces\`aro and $l^p$ spaces. This description is simpler than the one given by Jagers \cite{Ja74}. 
On the one hand, his factorization method was universal enough to be adopted to the function case on $[0,\infty)$ and $[0,1]$, 
which was done by Astashkin and Maligranda \cite[Theorem 3]{AM09}. However, this method is rather complicated 
and indirect, possibly valid only for power functions. 
Totally different approach appeared in the paper by Sy, Zhang and Lee \cite{SZL87}, but the idea was based on Jagers' result. 
Also Kami\'nska and Kubiak in \cite{KK12} were inspired by Jagers when they found an isometric representation of the K\"othe 
dual of weighted Ces\`aro function spaces $Ces_{p,w}$. They used not so easy Jagers' idea of relative concavity and concave 
majorants. In 2007 in the paper by Kerman, Milman and Sinnamon \cite{KMS07} the abstract Ces\`aro spaces $CX$
appeared and the K\"othe dual of these spaces was found but only in the case of rearrangement invariant space $X$ 
on $I = [0, \infty)$. The result comes from the equivalence (in norm) of the Ces\`aro operator with the so-called level function. 
This time again, one has to go through a very technical theory of down spaces and level functions to get the mentioned dual. 

Our goal in this paper is to grasp the general nature of the problem of duality of Ces\`aro spaces. We shall prove the duality 
theorem for abstract Ces\`aro spaces (isomorphic form) by as elementary method as possible. Moreover, our method applies to 
a very general case and in this case our proof is easier and more comprehensive than any earlier. Especially the function case on 
$[0,\infty)$ is instructive, but more delicate modification must be done in the case of interval $[0,1]$. Generally, for one inclusion 
there is crucial Sinnamon's result (which is however very intuitive and elementary one and avoids level functions) and 
the second inclusion follows from some kind of idempotency of the Ces\`aro operator, which was noticed firstly for the sequence 
case by Bennett in \cite{Be96}, which proof was simplified by Curbera and Ricker in \cite{CR13}.

The paper is organized as follows. In Section 2 some necessary definitions and notations are collected, together with some basic
results on Ces{\`a}ro abstract spaces. In particular, we can see when the abstract Ces{\`a}ro spaces $CX$ are nontrivial.

Sections 3 and 4 contain results on the K\"othe dual $(CX)^{\prime}$ of abstract Ces{\`a}ro spaces. There is a big difference 
between the cases on $[0, \infty)$ and on $[0, 1]$, as we can see in Theorems \ref{thm:dualgeneralR}, \ref{thm:dualgeneralon01b} 
and \ref{thm:dualgeneralon01a}. Important in our investigations were earlier results on 
the K{\"o}the dual $(Ces_p[0, \infty))^{\prime}$ due to Kerman-Milman-Sinnamon \cite{KMS07} and $(Ces_p(I))^{\prime}$ due 
to Astashkin-Maligranda \cite{AM09}.

In Section 5, the K\"othe dual of abstract Ces{\`a}ro sequence spaces is presented in Theorem \ref{thm:dualsequence}. We also collected here our knowledge about earlier results on K\"othe duality of Ces{\`a}ro sequence spaces $ces_p$ and their weighted versions. 

In Section 6 we first give in Theorem \ref{thm:LZ} a simple proof of a generalization of the Luxemburg-Zaanen \cite{LZ66} and 
Tandori \cite{Ta55} results on duality of weighted Ces{\`a}ro spaces $Ces_{\infty, w}$. This proof is also working for weighted Ces{\`a}ro 
sequence spaces (implicitely proved by Alexiewicz \cite{Al57}). Then in Theorem \ref{thm:DS} we identify the Ces{\`a}ro-Lorentz space 
$C\Lambda_{\varphi}$ with the weighted $L^1$-space, using our duality result proved in Theorem 3, which simplifies the result 
of Delgado and Soria \cite{DS07}.

Finally, in connection to the proof of Theorem \ref{thm:dualgeneralon01a} we collected Appendices A and B at the end of this paper. 
First one, is about weighted version of the Calder\'on-Mitjagin interpolation theorem and the second contains an improvement 
of the Hardy inequality in weighted spaces $L^p(x^{\alpha})$ on $[0, 1].$

\vspace{-3mm}

%%%%%%%%%%%%%%%%%%%%%%%%%%%%%% 2. Definitions
\section{ Definitions and basic facts}
We recall some notions and definitions which we will need later on. By $L^0 = L^0(I)$ we denote the set of all equivalence classes of
real-valued Lebesgue measurable functions defined on $I = [0, 1]$ or $I = [0, \infty)$. A {\it Banach ideal space} $X = (X, \|\cdot\|)$ (on $I$) is understood to be a Banach space in $L^0(I)$, which satisfies the so-called ideal property: if $f, g \in L^0(I), |f| \leq |g|$ a.e. on $I$ and 
$g \in X$, then $ f\in X$ and $\|f\| \leq \|g\|$. Sometimes we write $\|\cdot\|_{X}$ to be sure in which space the norm is taken. If it is not stated otherwise, then we understand that in a Banach ideal space there is $f\in X$ with $f(t) > 0$ for each $t\in I$ (such a function is called the {\it weak unit} 
in $X$), which means that ${\rm supp}X = I$.

Since the inclusion of two Banach ideal spaces is continuous, we should write $X\hookrightarrow Y$ rather that $X\subset Y$.
Moreover, the symbol $X\overset{A}{\hookrightarrow }Y$ means $X\hookrightarrow Y$ with the norm of inclusion not bigger 
than $A$, i.e., $\| f\|_{Y} \leq A \|f\|_{X}$ for all $f\in X$. Also $X = Y$ (and $X \equiv Y$) means that the spaces are the same 
and the norms are equivalent (equal).

For a Banach ideal space $X = (X, \|\cdot\|)$ on $I$ the {\it K{\"o}the dual space} (or {\it associated space}) $X^{\prime}$ is the space of
all $f \in L^0(I)$ such that the {\it associate norm}
\begin{equation} \label{dual}
\|f\|^{\prime}: = \sup_{g \in X, \, \|g\|_{X} \leq 1} \int_{I} |f(x) g(x) | \, dx
\end{equation}
is finite. The K{\"o}the dual $X^{\prime} = (X^{\prime}, \|\cdot \|^{\prime})$ is then a Banach ideal space.
Moreover, $X \overset{1}{\hookrightarrow }X^{\prime \prime}$ and we have equality 
$X = X^{\prime \prime}$ with $\|f\| = \|f\|^{\prime \prime}$ if and only if the norm in $X$ has
the {\it Fatou property}, that is, if the conditions $0 \leq f_{n} \nearrow f$ a.e. on $I$
and $\sup_{n \in {\bf N}} \|f_{n}\| < \infty$ imply that $f \in X$ and $\|f_{n}\| \nearrow \|f\|$.

For a Banach ideal space $X = (X, \|\cdot\|)$ on $I$ with the K{\"o}the dual $X^{\prime}$
there holds the following {\it generalized H{\"o}lder-Rogers inequality}: if $f \in X$ and $g \in X^{\prime}$,
then $f g$ is integrable and
\begin{equation} \label{HRinequality}
\int_{I} |f(x) g(x)|\, dx \leq \|f\|_{X} \|g\|_{X^{\prime}}.
\end{equation}
A function $f$ in a Banach ideal space $X$ on $I$ is said to have {\it order continuous norm} in $X$ if, for any
decreasing sequence of Lebesgue measurable sets $A_{n} \subset I $ with empty intersection, we have that 
$\|f \chi_{A_{n}} \| \rightarrow 0$ as $n \rightarrow \infty$. The set of all functions in $X$ with order
continuous norm is denoted by $X_{a}$. If $X_{a} = X$, then the space $X$ is said to be {\it order continuous}.
For order continuous Banach ideal space $X$ the K{\"o}the dual $X^{\prime}$ and the dual space $X^{*}$ coincide.
Moreover, a Banach ideal space $X$ with the Fatou property is reflexive if and only if both $X$ and its associate space 
$X^{\prime}$ are order continuous.

For a weight $w(x)$, i.e. a measurable function on $I$ with $0 < w(x) < \infty$ a.e. and for a Banach ideal space $X$ 
on $I$, the {\it weighted Banach ideal space} $X(w)$ is defined as $X(w)=\{f\in L^0: fw \in X\}$ with the norm 
$\|f\|_{X(w)}=\| f w \|_{X}$. Of course, $X(w)$ is also a Banach ideal space and 
\begin{equation} \label{dualweight}
[X(w)]^{\prime} \equiv X^{\prime}\Big(\frac{1}{w}\Big).
\end{equation}

By a {\it rearrangement invariant} or {\it symmetric space} on $I$ with the Lebesgue measure $m$, we mean a Banach ideal space $X=(X,\| \cdot \|_{X})$ with the additional property that for any two equimeasurable functions 
$f \sim g, f, g \in L^{0}(I)$ (that is, they have the same distribution functions $d_{f}\equiv d_{g}$, where 
$d_{f}(\lambda) = m(\{x \in I: |f(x)|>\lambda \}),\lambda \geq 0,$ and $f\in E$ we have $g\in E$ and $\| f\|_{E} = \| g\|_{E}$. In particular, 
$\| f\|_{X}=\| f^{\ast }\|_{X}$, where $f^{\ast }(t)=\mathrm{\inf } \{\lambda >0\colon \ d_{f}(\lambda ) < t\},\ t\geq 0$. 

For general properties of Banach ideal spaces and symmetric spaces we refer to the books \cite{BS88}, \cite{KA77}, \cite{KPS82}, 
\cite{LT79} and \cite{Ma89}.

%We will investigate abstract Ces\`aro space and $\widetilde{X}$ space. 
In order to define and formulate the results we need
the continuous Ces\`aro operator $C$ defined for $0 < x \in I$ as
$$
Cf(x) =\frac{1}{x} \int_0^x f(t) \,dt,
$$
and also a nonincreasing majorant $\widetilde{f}$ of a given function $f$, which is defined for $x \in I$ as
$$
\widetilde{f}(x) = {\rm ess} \sup_{t \in I, \, t \geq x} |f(t)|.
$$
For a Banach ideal space $X$ on $I$ we define an {\it abstract Ces\`aro space} $CX = CX(I)$ as 
\begin{equation} \label{Cesaro}
CX=\{f\in L^0(I): C|f| \in X\} ~~ {\rm with ~the ~norm } ~~ \|f\|_{CX} = \| C|f| \|_{X},
\end{equation}
and the space $\widetilde{X} = \widetilde{X} (I)$ as 
\begin{equation} \label{falka}
\widetilde{X}=\{f\in L^0(I): \widetilde{f}\in X\} ~~ {\rm with ~the ~norm } ~~ \|f\|_{\widetilde{X}}=\|\widetilde{f}\|_{X}.
\end{equation}

The space $CX$ for a Banach ideal space $X$ on $[0, \infty)$ was defined already in \cite{Ru80} and spaces $CX, \,\widetilde{X}$ 
for $X$ being a symmetric space on $[0,\infty)$ have appeared, for example, in \cite{KMS07} and \cite{DS07}.

The {\it dilation operators} $\sigma_\tau$ ($\tau > 0$) defined on $L^0(I)$ by 
$$
\sigma_\tau f(x) = f(x/\tau) \chi_{I}(x/\tau) = f(x/\tau) \chi_{[0, \, \min(1, \, \tau)]}(x), ~~ x \in I,
$$
are bounded in any symmetric space $X$ on $I$ and $\| \sigma_\tau \|_{X \rightarrow X} \leq \max (1, \tau)$ (see \cite[p. 148]{BS88} and 
\cite[pp. 96-98]{KPS82}). They are also bounded in some Banach ideal spaces which are not necessary symmetric. For example, if either $X = L^p(x^{\alpha})(I)$ or $X = CL^p(x^{\alpha})(I)$, then $\| \sigma_\tau\|_{X \rightarrow X} = \tau^{1/p + \alpha}$ 
(see \cite{Ru80} for more examples).

Let us collect some basic properties of $CX$ and $\widetilde{X}$ spaces. 

%%%%%%%%%%%%%%%%%%%%% Theorem 1
\begin{theorem}
Let $X$ be a Banach ideal space on $I$. Then both $CX$ and $\widetilde{X}$ are also Banach ideal spaces on $I$ (not necessary with a weak unit). Moreover,
\begin{itemize}
\item[$(a)$] $CX[0, \infty) \neq \{0\}$ if and only if $\frac{1}{x}\chi_{[a,\infty)}(x)\in X$ for some $a > 0$. 
\item[$(b)$] $CX[0, 1] \not =\{0\}$ if and only if $\chi_{[a,1]}\in X$ for some $0 < a <1$. 
\item[$(c)$] $\widetilde{X}\not =\{0\}$ if and only if $X$ contains a nonzero, nonincreasing function on $I$. 
\item[$(d)$] If $X$ has the Fatou property, then $CX$ and $\widetilde{X}$ have the Fatou property.
\item[$(e)$] $(\widetilde{X})_a = \{0\}$.
\end{itemize}
\end{theorem}

\proof (a). Suppose that $\frac{1}{x}\chi_{[a,\infty)}(x)\in X$ for some $a > 0$. For any $b > a$ we have 
\begin{eqnarray*}
\| \chi_{[a, b]} \|_{CX} 
&=& 
\| \frac{1}{x} \int_0^x \chi_{[a, b]}(t)\, dt \|_X =
\| \frac{x-a}{x} \chi_{[a, b]}(x) +  \frac{b-a}{x} \chi_{(b, \infty)}(x) \|_X \\
&\leq& 
 \| \frac{b-a}{x} \chi_{[a, \infty)}(x) \|_X = (b-a) \| \frac{1}{x} \chi_{[a, \infty)}(x) \|_X < \infty,
\end{eqnarray*}
whence $\chi_{[a, b]} \in CX$.

If $CX \neq \{0\}$, then there exists $0 \neq f \in CX$, that is, $|f(x)| > 0$ for $x \in A$ with $0 < m(A) < \infty$, and we 
can find $a > 0$ such that $b = \int_0^a |f(t)| \, dt > 0$. Thus, 
\begin{eqnarray*}
\frac{b}{x} \chi_{[a, \infty)}(x)
&\leq&
\frac{1}{x}  \int_0^a |f(t)| \, dt \, \chi_{[a, \infty)}(x) \\
&\leq&
\frac{1}{x}  \int_0^x |f(t)| \, dt \, \chi_{[a, \infty)}(x) \leq C|f|(x) \in X,
\end{eqnarray*}
and so $\frac{1}{x} \chi_{[a, \infty)}(x)  \in X$.

(b) Proof is similar as in (a). However, observe that the condition $\chi_{[a,1]}\in X$ has no weight $w(x) = 1/x$ as in (a).
Proof of (c) is clear. Proof of (d) follows from the facts that if $f_n, f \in L^0(I)$ and $0 \leq f_n \nearrow f$ pointwise on $I$, then
$Cf_n \nearrow Cf$ pointwise on $I$. Also $f_n \chi_{[x, \infty) \cap I} \nearrow f \chi_{[x, \infty) \cap I}$ for every $x \in I$, and so 
$\widetilde{f_n} \nearrow \widetilde{f}$ pointwise on $I$, which implies, by the Fatou property of $X$, that $\| f_n \|_{CX} = \| C f_n \|_{X} \nearrow \| C f \|_{X} 
= \| f \|_{CX}$ and also $\| f_n \|_{\widetilde{X}} = \| \widetilde{f_n}\|_X \nearrow  \| \widetilde{f} \|_X = \| f \|_{ \widetilde{X}}$ .

(e) Suppose $0\not = f \in \widetilde{X}$. Then ${\rm ess} \sup_{x \in I}|f(x)| = a > 0$. It means 
$$
m(\{x \in I: |f(x)| > a/2 \}| = b > 0.
$$ 
In particular, for $A = \{x \in I: |f(x)| > a /2\} \backslash [0,b /2]$ we have $m(A) \geq b /2$. Now, choose a sequence of sets of positive 
measure $(A_n)$ such that $m(\bigcap_{n=1}^{\infty}A_n)=0$, $A_{n+1}\subset A_n \subset A$ for each $n=1, 2, 3, \dots $. Then
$$
\frac{a}{2} \chi _{[0, b/2]}\leq \widetilde{f \chi _{A_n}}, ~{\rm for ~ all} ~~ n=1,2,3, \ldots,
$$ 
and consequently 
$$
\|\frac{a}{2} \chi _{[0, b /2]}\|_X = \|\frac{a}{2} \chi _{[0, b /2]}\|_{\widetilde{X}} \leq \|\widetilde{f \chi _{A_n}}\|_X = \| f \chi _{A_n}\|_{\widetilde{X}},
$$
which means that $f \not \in (\widetilde{X})_a$.
\endproof

%%%%%%%%%%%% Remark 1
\begin{remark}
It is important to notice that there are Banach ideal spaces $X$ for which Ces\`aro space $CX \neq \{0\}$ but it does not contain a weak unit (see Example \ref{przyklad} below). Using the notion of support of the space (more information about this notion can 
be found, for example, in \cite[p. 137]{KA77}, \cite[pp. 169-170]{Ma89} and \cite[pp. 879-880]{KLM13}) we can say even more, 
namely that ${\rm supp} CX \subset {\rm supp}X$ and the inclusion can be essentially strict. On the other hand, in our investigations of the duality of Ces\`aro spaces the 
natural assumption is that Ces\`aro operator is bounded in the given Banach ideal space $X$ which ensures that 
${\rm supp} CX = {\rm supp} X = I$.
\end{remark}

%%%%%%%%%%%%%%%%%% Example 2
\begin{example}\label{przyklad}
Consider the Banach ideal space $X=L^p(w)$ with $1<p<\infty $ and weight  $w(x) = \max (\frac{1}{1-x},1)$ on $I = [0, \infty)$. Then 
${\rm supp} X = I$ but ${\rm supp} CX = [1, \infty)$. 
\end{example}

%%%%%%%%%%%% Remark 3
\begin{remark}
Ces\`aro spaces $CX$ on $I$ are not symmetric spaces even when $X$ is a symmetric Banach space on $I$ and Ces\`aro operator 
$C$ is bounded on $X$. In fact, it was proved in \cite[Theorem 2.1]{DS07} that $CX \not \hookrightarrow L^1 + L^{\infty}$ for $I = [0, \infty)$ 
 from which it follows that $CX[0, \infty)$ cannot be symmetric and in \cite{AM09} it was shown that $CL^p [0, 1]$ is not symmetric. It seems that $CX$ is never symmetric (even cannot be renormed to be symmetric) but this is only our conjecture.
\end{remark}

%%%%%%%%%%%%%%%%%%%%%%%%%%%%%%%%%%%%%%%% Section 3
\section{Duality on $[0,\infty)$}

The description of K\"othe dual spaces of $Ces_{p}[0, \infty)$ spaces for $1 < p \leq \infty$ appeared as remark in Bennett \cite[p. 124]{Be96}, but it was proved by Astashkin-Maligranda \cite{AM09}. For more general spaces $CX$, where $X$ is a symmetric space having additional properties, it was proved by Kerman-Milman-Sinnamon \cite[Theorem D]{KMS07} and they used in the proof some of Sinnamon's results \cite[Theorem 2.1]{Si03} and \cite[Proposition 2.1 and Lemma 3.2]{Si01}.

In the case of symmetric spaces on $[0,\infty)$ one can simplify the proof of duality theorem from \cite[Theorem D]{KMS07}, using 
the results of Sinnamon (cf. \cite{Si01}, \cite{Si03} and \cite{Si07}). The proof below seems to be simpler, although uses the same ideas. 

%%%%%%%%%%%Theorem 2
\begin{theorem}\label{thm:dual}
Let $X$ be a symmetric space on $[0,\infty)$ with the Fatou property. If $C$ is a bounded operator on $X$, then 
\begin{equation} \label{dual1}
(CX)^{\prime} = \widetilde{X^{\prime}}.
\end{equation}
\end{theorem}

%%%%%%%%%%Proof
\proof
From Theorem 1 spaces $CX$ and $\widetilde{X^{\prime}}$ have the Fatou property. Therefore it is enough to show that 
$$
CX = (CX)^{\prime \prime} = \big (\widetilde{X^{\prime}} \big)^{\prime}.
$$
For $f \in \left(\widetilde{X^{\prime}} \right)^{\prime}$ we have
\begin{eqnarray*}
\|f\|_{(\widetilde{X^{\prime}})^{\prime}} 
&=& 
\sup\{ \int_I | f(x) \, g(x)| \, dx: \, g \in \widetilde{X^{\prime}}, \|g\|_{\widetilde{X^{\prime}}}\leq 1 \} \\
&=& 
\sup\{ \int_I | f(x) \, g(x)| \, dt: \,\widetilde{g} \in X^{\prime}, \| \widetilde{g}\|_{X^{\prime}}\leq 1 \} : = I(f).
\end{eqnarray*}
From one side
$$
 I(f) \leq \sup\{ \int_I | f(x)| \, \widetilde{g}(x) \, dx: \, \widetilde{g} \in X^{\prime}, \| \widetilde{g} \|_{X^{\prime}}\leq 1 \}, 
$$
and on the other hand
\begin{equation*}
I(f) \geq 
\sup\{ \int_I | f(x) \, g(x)| \, dx: \, g = \widetilde{g} \in X^{\prime}, \| \widetilde{g} \|_{X^{\prime}}\leq 1 \},
\end{equation*}
because the supremum on the right is taken over smaller set of functions. Therefore, 
\begin{equation} \label{down}
\|f\|_{(\widetilde{X^{\prime}})^{\prime}}  = \sup\{ \int_I | f(x)| \, h(x) \, dt: 0\leq h \downarrow, \|h\|_{X^{\prime}}\leq 1 \} = : \|f\|_{X^{\downarrow}}.
\end{equation}
The last supremum define the so-called ``down space" $X^{\downarrow}$ (cf. \cite{Si94} and \cite{Si01}) and so
\begin{equation} \label{33}
\|f\|_{(\widetilde{X^{\prime}})^{\prime}} = \|f\|_{X^{\downarrow}}.
\end{equation}
For $I = [0, \infty)$ the identification of the down space $X^{\downarrow}$ with the Ces\`aro space $CX$, provided operator $C$ 
is bounded on the symmetric space $X$, is a consequence of Sinnamon's result \cite[Theorem 3.1]{Si01}. Thus,
\begin{equation} \label{34}
\|f\|_{(\widetilde{X^{\prime}})^{\prime}} = \|f\|_{X^{\downarrow}} \approx \|C|f| \|_{X} = \| f\|_{CX}.
\end{equation}
\endproof

%%%%%%%%%%%%% Remark 4
\begin{remark}
Let us mention that the duality of down space $X^{\downarrow}$ for symmetric space $X$ with the help of level functions was investigated by Sinnamon in his papers \cite[Theorem 6.7]{Si94}, \cite[Theorem 5.7]{Si01} and \cite[Theorem 2.1]{Si07}. Another proof of (\ref{33}) can be found in \cite[Theorem 5.6]{Si01}.
\end{remark}

Let us generalize the above theorem to a wider class than symmetric spaces, which will include corresponding isomorphic versions from 
\cite{KK12} and \cite{Ja74}. Our method here is more direct and does not need neither the notion of level functions nor down spaces. 
However, one result of Sinnamon, namely \cite[Theorem 2.1]{Si03} (see also \cite{Si07} for a very nice intuitive graphical explanation 
of this equality) will be necessary. Since the result is given for a general measure on $\mathbb{R}$, we need to reformulate it slightly 
to make it compatible with our notion. 

%%%%%%%%%%%%%%%%%%%%% Proposition
\begin{proposition}[Sinnamon, 2003] \label{Sin}
Let either $I=[0,\infty)$ or $I=[0,1]$. For a measurable $f, g, h \geq 0$ on $I$ we have
\begin{equation} \label{Sinn}
\int_{I} f(x)\, \widetilde{g}(x) \,dx = \sup_{h \prec f} \int_{I}h(x) \,g(x) \,dx,
\end{equation}
where ${h \prec f}$ means that $\int_{0}^u h(x)dx \leq \int_{0}^u f(x)dx$ for all $u \in I$.
\end{proposition} 

%%%%%%%%
\proof
From \cite[Theorem 2.1]{Si03} we have for a $\lambda$ - measurable $f, g, h \geq 0$ on $\mathbb{R}$ 

\begin{equation*}
\int_{\mathbb{R}} f\, \widetilde{g} \,d\lambda = \sup_{h \prec_{\lambda} f} \int_{\mathbb{R}} h \,g \,d\lambda,
\end{equation*}
where $h \prec_{\lambda} f$ means that $\int_{-\infty}^u hd\lambda\leq \int_{-\infty}^u fd\lambda $ for all $u \in \mathbb{R}$. 
In the case $I = [0, \infty)$ we put $\lambda$ to be just the Lebesgue measure on $[0,\infty)$ and zero elsewhere. 
Then $h \prec_{\lambda} f$ if and only if $h \prec f$ because for each $u \geq 0$
$$
\int_{-\infty}^u hd\lambda\leq \int_{-\infty}^u fd\lambda  \Longleftrightarrow \int_{0}^u h(x)dx\leq \int_{0}^u f(x)dx. 
$$
Then 
$$
\int_{0}^{\infty} f(x) \widetilde{g}(x) \,dx = \int_{\mathbb{R}} f\, \widetilde{g} \,d\lambda = \sup_{h \prec_{\lambda} f} \int_{\mathbb{R}} h \,g \,d\lambda=\sup_{h \prec f} \int_{0}^{\infty} h(x) \,g(x) \,dx.
$$
In the case of interval $[0,1]$ we put $\lambda$ to be the Lebesgue measure on $[0,1]$ and zero elsewhere. Then
$\widetilde{g}(x)=\widetilde{g\chi_{[0,1]}}(x)$ for $x\in [0,1]$ and the remaining part of the proof works in the same way as before.  
\endproof

%%%%%%%%%%%%%%%%%%%%%%%%%% Theorem 3
\begin{theorem}\label{thm:dualgeneralR}
Let $X$ be a Banach ideal space on $I = [0,\infty)$ such that both the Ces\`aro operator $C$ and the dilation operator $\sigma_\tau$, for 
some $0 < \tau < 1$, are bounded on $X$. Then 
\begin{equation} \label{Thm2}
(CX)^{\prime} = \widetilde{X^{\prime}}  ~~~~ {\it with ~equivalent ~norms.}
\end{equation}
\end{theorem}

We start with the continuous version of the inequality, proved for sequences by Curbera and Ricker \cite[Proposition 2]{CR13}.

%%%%%%%%%%%%%%%%%%%%Lemma 6 
\begin{lemma}\label{prop:curbera}
If $0 \leq f\in L_{loc}^1[0,\infty)$ and $a > 1$ is arbitrary, then
\begin{equation} \label{Lemat1}
\int_0^{x/a} f(t) \, dt \leq \frac{1}{\ln a} \int_0^x \left(\frac{1}{t} \int_0^t f(s)\, ds\right) dt ~~{\it for ~all} ~~x > 0,
\end{equation}
that is, $Cf(x/a) \leq \dfrac{a}{\ln a} \, CCf(x)$ for all $x > 0$.
\end{lemma}

%%%%%%%%%Proof-Lemma
\proof
For $x > 0$, by the Fubini theorem, we have
\begin{eqnarray*}
\int_0^x \left(\frac{1}{t} \int_0^t f(s)\, ds\right) dt 
&=& 
\int_0^x  f(s) \left(\int_s^x \frac{1}{t}\, dt \right) ds = \int_0^x  f(s) \ln\frac{x}{s}\, ds \\
&=& 
\int_0^{x/a} f(s) \ln\frac{x}{s} \, ds + \int_{x/a}^x f(s) \ln\frac{x}{s} \, ds \\
&\geq& 
\ln a \, \int_0^{x/a} f(s) \, ds,
\end{eqnarray*}
and so 
$$  
CCf(x) = \frac{1}{x} \int_0^x Cf(t) \, dt \geq \dfrac{\ln a}{a} \, Cf(x/a).
$$
\endproof

%%%%%%%%%%%%% Remark 7
\begin{remark}
From the classical Hardy inequality and (\ref{Lemat1}) we obtain that if $1 < p \leq \infty$, then 
$$
C^2L^p = C(CL^p) = CL^p.
$$
Such an equality for Ces\`aro sequence spaces $ces_p \, (1 < p < \infty)$ was proved by Bennett (cf. \cite[Theorem 20.31]{Be96}) 
and simplified in \cite{CR13}.
In fact, by the Hardy inequality we have $L^p \overset{p^{\prime}}{\hookrightarrow } CL^p$ 
(cf. \cite{KMP07}) and so $CL^p \overset{p^{\prime}}{\hookrightarrow } CCL^p$. On the other hand, (\ref{Lemat1}) shows that
$$
a^{1/p} \| C|f| \|_{L^p} = \| \sigma_{a}(C|f|)\|_{L^p} \leq \dfrac{a}{\ln a} \| CC|f| \|_{L^p}.
$$
Since $\inf_{a > 1} \dfrac{a^{1-1/p}}{\ln a} = \dfrac{e}{p^{\prime}}$, it follows that $CCL^p \overset{e/p^{\prime}}{\hookrightarrow } CL^p$.
\end{remark}

%%%%%%%%%%%Proof of Theorem 3
\proof[Proof of Theorem \ref{thm:dualgeneralR}]
We shall start with the usually simpler inclusion $\widetilde{X^{\prime}} \hookrightarrow (CX)^{\prime}$. Taking the substitution 
$t = a u$ in the right integral of (\ref{Lemat1}) we obtain
\begin{eqnarray*}
\int_0^{x/a} |f(t)| \, dt 
&\leq& 
\frac{a}{\ln a} \int_0^{x/a} \left(\frac{1}{au}\int_0^{au} |f(s)| \, ds\right) du \\
&=& 
\frac{a}{\ln a} \int_0^{x/a} C|f|(a u) \, du
~ ~~{\rm for ~ all} ~~  x>0. 
\end{eqnarray*}
Let $g\in \widetilde{X'}$ and $f \in CX$, then applying the above estimate to the property $18^o$ from page 72 in \cite{KPS82} and by the H\"older-Rogers inequality (\ref{HRinequality}) we obtain
\begin{eqnarray*}
\int_0^{\infty} |f(x) \,g(x)| \, dx
&\leq&
\int_0^{\infty} |f(x)| \,\widetilde{g}(x) \, dx \leq \frac{a}{\ln a} \int_0^{\infty} \,C|f|(ax) \, \widetilde{g}(x)\, dx\\
&\leq&
\frac{a}{\ln a} \| C|f|(ax)\|_X \| \widetilde{g} \|_{X^{\prime}} \leq \frac{a}{\ln a} \| \sigma_{1/a}\|_{X \rightarrow X} \| C|f| \|_X \| \widetilde{g} \|_{X^{\prime}} \\
&=&
\frac{a}{\ln a} \| \sigma_{1/a}\|_{X \rightarrow X} \| f \|_{CX} \| g \|_{\widetilde{X^{\prime}}} ,
\end{eqnarray*}
which means that 
$$
\|g\|_{(CX)^{\prime}}\leq \frac{a}{\ln a} \|\sigma_{1/a}\|_{X\rightarrow X} \|g\|_{\widetilde{X^{\prime}}}
$$
and so $\widetilde{X^{\prime}}  \overset{A}{\hookrightarrow } (CX)^{\prime}$ with $A = \frac{a}{\ln a} \|\sigma_{1/a}\|_{X\rightarrow X}$ 
and $a > 1$.

We can now turn our attention into, usually more difficult, second inclusion $(CX)'\hookrightarrow \widetilde{X^{\prime}}$. 
Let $g\in (CX)^{\prime}$ and $f \in X$, then by (\ref{Sinn}) in Proposition \ref{Sin}, the generalized H\"older-Rogers inequality 
and using the assumption that operator $C$ is bounded on $X$ we obtain
\begin{eqnarray*}
\int_{0}^{\infty} |f(x)| \, \widetilde{g}(x) \,dx 
&=& 
\sup_{|h| \prec |f|} \int_{0}^{\infty} |h(x) g(x)| \,dx \leq \sup_{|h|\prec |f|}\|h\|_{CX}\| g \|_{(CX)^{\prime}} \\
&\leq& 
\|f\|_{CX} \| g \|_{(CX)^{\prime}} = \| C|f| \|_{X} \| g \|_{(CX)^{\prime}}\leq B \|f\|_{X}\|g\|_{(CX)^{\prime}},
\end{eqnarray*}
where $B = \|C\|_{X\rightarrow X}$. Therefore, 
$$
\|g\|_{\widetilde{X^{\prime}}}=\|\widetilde{g}\|_{X^{\prime}}=\sup_{\|f\|_{X}\leq 1} \int_{0}^{\infty} |f(x)|\, \widetilde{g}(x) \,dx \leq B \|g\|_{(CX)^{\prime}}
$$
and thus $(CX)^{\prime}  \overset{B}{\hookrightarrow } \widetilde{X^{\prime}}$.   
\endproof

%%%%%%%%%%%%% Remark 8
\begin{remark}
It is worth to notice that in Theorem \ref{thm:dualgeneralR}, the assumption on boundedness of dilation operator was used only for 
the inclusion $ \widetilde{X^{\prime}}  \hookrightarrow (CX)^{\prime}$. On the other hand, the proof of inclusion $(CX)^{\prime} \hookrightarrow \widetilde{X^{\prime}}$ requires only boundedness of $C$ on $X$.
\end{remark}

Some authors considered weighted Ces\`aro operators or Ces\`aro operator $C$ in weighted $L^p(w)$ spaces which leads 
to weighted Ces\`aro spaces $Ces_{p, w}$. These spaces are particular examples of abstract Ces\`aro spaces $CX$, in fact,
$Ces_{p, w} = C\big(L^p(w)\big)$. From Theorem \ref{thm:dualgeneralR} we obtain the following duality result even for more general 
weighted Ces\`aro spaces $C\big(X(w)\big)$.

%%%%%%%%%%% Corollary 9
\begin{corollary}
Let $X$ be a symmetric space on $[0, \infty)$ and $w$ be a weight on $[0, \infty)$ such that the dilation operator $\sigma_a$ (for some 
$0 < a < 1$) and Ces\`aro operator $C$ are bounded on $X(w)$. Then 
$$
\big[C\big(X(w)\big)\big]^{\prime} = \widetilde{\big( X^{\prime}(\frac{1}{w})\big)}.
$$
\end{corollary}
It is not clear if our approach includes all weights from \cite{KK12} (with equivalent norms) but for the power weight $w(x) = x^{\alpha}$ 
with $\alpha < 1-1/p$ and $1 < p< \infty$ we obtain
$$
[Ces_{p, x^{\alpha}}]^{\prime} = \big[C\big(L^p(x^{\alpha})\big)\big]^{\prime} =  \widetilde{L^{p^{\prime}}(x^{-\alpha})},
$$
since $C$ is bounded in $L^p(x^{\alpha})$ with the norm $\| C \| = (1 - \alpha - 1/p)^{-p}$ (cf. \cite[p. 245]{HLP52} or \cite[p. 23]{KMP07})
and $\sigma_\tau$ has norm $\| \sigma_\tau\| = \tau^{1/p + \alpha}$.

%The case $I = [0, 1]$ is essentailly different it will be considered in the next Section.

%%%%%%%%%%%%%%%%%%%%%%%%%%%%%%%%%%% Section 4
\section{\protect \medskip Duality on $[0,1]$}

The duality of Ces\`aro spaces on $I = [0,1]$ is more delicate and less known. Astashkin-Maligranda \cite[Theorem 3]{AM09} 
proved that for $1 < p < \infty$ we have $(Ces_p)^{\prime} = U(p^{\prime}): = \widetilde {L^{p^{\prime}}(\frac{1}{1-x})}$, where 
$f \in \widetilde {L^{p^{\prime}}(\frac{1}{1-x})}$ means that $\widetilde{f} \in L^{p^{\prime}}(\frac{1}{1-x})$ with the norm
$$
\| f\|_{U(p^{\prime})} = \left[ \int_0^1\Big(\frac{ \widetilde{f}(x)}{1-x} \Big)^{p^{\prime}} dx \right]^{1/p}.
$$
The proof of inclusion $U(p^{\prime}) {\hookrightarrow } (Ces_p)^{\prime}$ required improvement of the Hardy inequality, which 
they gave in \cite[inequality (21)]{AM09}: if $1 < p < \infty$, then $C: L^p(1-x) \rightarrow L^p$ is bounded, that is,
\begin{equation} \label{AM}
\| Cf \|_{L^p} \leq A_p \, \| (1-x) f(x) \|_{L^p} ~~ {\rm for ~all} ~~ f \in L^p(1-x),
\end{equation}
with $A_p \leq 2 (p^{\prime} + 2p)$. Their proof gives even more general result, which we will use later on. Let us present proof 
for the case of symmetric spaces on $[0, 1]$. We will need in the proof Copson operator which is defined by formula 
$C^*f(x) = \int_x^1 \frac{f(t)}{t} \,dt$.

%%%%%%%%%%%%%%%%%%%%% Lemma 10
\begin{lemma}[Astashkin-Maligranda, 2009] \label{AM09}
If $X$ is a symmetric space on $I = [0, 1]$ and both operators $C, C^*: X \rightarrow X$ are bounded, 
then $C: X(1-x) \rightarrow X$ is also bounded.
\end{lemma}

\begin{proof}
Let $w(x) = 1 - x, f(x) \geq 0$ for $x \in I$. Then for $0 < x \leq 1/2$
$$
Cf(x) = \frac{1}{x} \int_0^x f(t)\, dt \leq \frac{2}{x} \int_0^x f(t) w(t)\, dt = 2 \, C(fw)(x),
$$
and for $1/2 \leq x \leq 1$
\begin{eqnarray*}
Cf(x)
&\leq&
2 \int_0^x f(t)\, dt = 2 \int_{1-x}^1 f(1-t)\, dt \\
&=& 
2 \int_{1-x}^1 \frac{f(1-t) w(1-t)}{t}\, dt = 2 \, C^*\big( ~ \overline {fw} ~ \big)(1-x),
\end{eqnarray*}
where $ \overline {fw} (t) = f(1-t) w(1-t)$. Therefore,
$$
\| Cf \|_X \leq 2 \, \| C(fw) \|_X + 2 \, \| C^*(\overline {fw}) \|_X \leq 2 \, ( \| C \|_{X \rightarrow X} + \| C^* \|_{X \rightarrow X}) \, \| f w \|_X.
$$
\end{proof}

We shall prove that $(CX)^{\prime} = \widetilde {X^{\prime}(\frac{1}{1-x})}$ under some assumptions on the space $X$ but each inclusion 
will be proved separately.

%%%%%%%%%%%%%%%%%%%%%%%%% Theorem 4
\begin{theorem}\label{thm:dualgeneralon01b}
Let $X$ be a Banach ideal space on $I = [0,1]$ such that the operator $C: X(1-x) \rightarrow X$ is bounded. Then 
\begin{equation} \label{Thm4}
(CX)^{\prime} \hookrightarrow \widetilde {X'\big(\frac{1}{1-x}\big)}.
\end{equation}
\end{theorem}

%%%%%%%%% Proof
\proof
Let $g\in (CX)^{\prime}$ and $f \in CX$, then using Proposition \ref{Sin} and applying our assumption in the last inequality we get
\begin{eqnarray*}
\int_0^1 |f(x)| \, \widetilde{g}(x) \,dx 
&=& 
\sup_{|h| \prec |f|} \int_0^1 |h(x) g(x)| \,dx \leq \sup_{|h|\prec |f|}\|h\|_{CX}\| g \|_{(CX)^{\prime}} \leq \|f\|_{CX} \| g \|_{(CX)^{\prime}} \\
&=& 
\| Cf \|_{X} \| g \|_{(CX)^{\prime}}\leq D \| (1-x) f(x) \|_{X}\|g\|_{(CX)^{\prime}},
\end{eqnarray*}
where $D = \|C\|_{X(1-x)\rightarrow X}$. Since $[X(1-x)]^{\prime} \equiv X^{\prime}(\frac{1}{1-x})$, it follows that
$$
\|g\|_{\widetilde{X'(\frac{1}{1-x})}} = \sup_{\|f\|_{X(1-x)}\leq 1} \int_0^1 |f(x) \, \widetilde{g}(x)| \,dx \leq D\, \|g\|_{(CX)^{\prime}},
$$
thus $(CX)^{\prime} \overset{D}{\hookrightarrow } \widetilde{X^{\prime}(\frac{1}{1-x})}$ with $D = \|C\|_{X(1-x)\rightarrow X}$.   
\endproof

%%%%%%%%%%%%%%%%%%%%%%%% Theorem 5
\begin{theorem}\label{thm:dualgeneralon01a}
If $X$ is a symmetric space on $I=[0,1]$ with the Fatou property, then  
\begin{equation} \label{Thm5}
\widetilde{X^{\prime} \big(\frac{1}{1-x}\big)}\hookrightarrow (CX)^{\prime}.
\end{equation}
\end{theorem}

In the proof of Theorem  \ref{thm:dualgeneralon01a} we will need the following result, similar to Lemma \ref{prop:curbera}.

%%%%%%%%%%%%%%%%%%%%% Lemma 11
\begin{lemma}\label{lem:curberaon01}
If $\int_0^t | f(x)| dx <\infty$ for each $0<t<1$, then
$$
\int_0^{t/d(t)} |f(x)| dx \leq \int_0^t \frac{1}{1-x} \left(\frac{1}{x} \int_0^x |f(s)| ds \right) dx = \int_0^t \frac{C|f|(x)}{1-x} dt ~~ {\it for ~all} ~ 0 < t < 1,
$$
where $d(t):= t+e-et$.
\end{lemma}

%%%%%%%%%%
\proof
Observe that $1 < d(t) < e$ for $0 < t < 1$. By the Fubini theorem we obtain
\begin{eqnarray*}
\int_0^t \frac{C|f|(x)}{1-x} \, dx 
&=& 
\int_0^t \frac{\int_0^x |f(s)| \, ds}{x(1-x)} \, dx = \int_0^t |f(s)| \left(\int_s^t \frac{dx}{x(1-x)} \right) ds \\
&=& 
\int_0^t |f(s)| \left[\int_s^t (\frac{1}{x}  + \frac{1}{1-x}) \, dx \right] ds = \int_0^t |f(s)| \, \ln\frac{t(1-s)}{s(1-t)} \, ds \\
&=&
\int_0^{t/d(t)} |f(s)| \ln\frac{t(1-s)}{s(1-t)} \, ds + \int_{t/d(t)}^t |f(s)| \ln\frac{t(1-s)}{s(1-t)} \, ds.
\end{eqnarray*}
It is easy to see that for $0 <s \leq t/(t+e-et)$ we have $\frac{t(1-s)}{s(1-t)}\geq e$ and of course $\ln \frac{t(1-s)}{s (1-t)} \geq 1$ for each $0<s<1$ so we get  
$$
\int_0^{t/d(t)} |f(x)| \, dx\leq \int_0^t \frac{C|f|(x)}{1-x} \, dx
$$
as required.
\endproof

%%%%%%%%%%%%%% Proof of Theorem 5
\proof[Proof of Theorem \ref{thm:dualgeneralon01a}]
Let $0\leq g=\widetilde{g}\in X^{\prime}(\frac{1}{1-x})$ be a simple function. Since $g$ is nonincreasing, $X^{\prime} \hookrightarrow L^1$ 
and $\frac{1}{1-x}\not\in L^1$, we can see that $g$ may be written as 
$$
g=\sum_{k=1}^n a_k \chi_{[0, \,t_k)},
$$
for $0\leq a_k$ and $0 < t_k < 1$ for $k=1,2,...,n$. Define $S$ as a class of all  $g$ of the above form. 
We need to show that there is a constant $M > 0$ such that for each $f\in CX$ and $g \in S$
\begin{equation}
\int_0^1|g(x) f(x)| \, dx \leq M \big \| \frac{g(x)}{1-x} \big \|_{X^{\prime}} \, \big \|C|f| \big\|_X = M \|g\|_{X^{\prime}(\frac{1}{1-x})} \|f\|_{CX}.\label{nierHold1}
\end{equation}
For $d(t)=t+e-et$ denote 
$$
g_d=\sum_{k=1}^n a_k \chi_{[0, \, t_k/d(t_k))}.
$$
Then, by Lemma \ref{lem:curberaon01}, we get 
\begin{eqnarray*}
\int_0^1g_d(x) |f(x)| \, dx 
&=& 
\sum_{k=1}^n a_k \int_0^{t_k/d(t_k)} |f(x)| \, dx \\
&\leq& 
\sum_{k=1}^n a_k \int_0^{t_k} \frac{C|f|(x)}{1-x} \,dx = \int_0^1g(x) \frac{C|f|(x )}{1-x} \, dx.
\end{eqnarray*}
Applying generalized H\"older-Rogers inequality (\ref{HRinequality}) one has 
$$
\int_0^1g_d(x) |f(x)| \, dx \leq \int_0^1g(x)\frac{C|f|(x)}{1-x} \, dx \leq \big \|\frac{g(x)}{1-x} \big\|_{X^{\prime}} \, \big\| C|f| \big\|_X.
$$
Therefore, to get (\ref{nierHold1}) we need to find a constant $M>0$, independent of the choice of $g$, such that 
$\|g\|_{X^{\prime}(\frac{1}{1-x})}\leq M\|g_d\|_{X^{\prime}(\frac{1}{1-x})}$. To do this, let us consider a function 
$\sigma: [0,1] \rightarrow [0,1]$ given by $\sigma(t)=\frac{t}{d(t)}=\frac{t}{t+e-et}$. Then $\sigma^{-1}(t)=\frac{et}{1-t+et}$.
Define the composition operator $T$ on $L^0[0, 1]$ by
$$
Th(t) = h(\sigma(t)).
$$
The key now is to notice that 
$$
T\chi_{[0, \, a)}(t)=\chi_{[0, \, a)}(\sigma(t))=\chi_{[0, \, \sigma^{-1}(a))}(t),
$$
where the last equality is a consequence of equivalence $\sigma(t)=a \Leftrightarrow \sigma^{-1}(a)=t$. 
If therefore $a=x/d(x)$, then 
$$
T\chi_{[0, \, x/d(x))} = \chi_{[0, \, x)}
$$ 
and consequently for $g$ and $g_d$ like above
$$
Tg_d=g.
$$
To complete the proof it is enough to show that $T$ is bounded on $X^{\prime}(\frac{1}{1-x})$. Of course, by the weighted version 
of Calder\'on-Mitjagin interpolation theorem (cf. Appendix A) it is enough to prove its boundedness only on $L^{\infty}(\frac{1}{1-x})$ 
and on $L^1(\frac{1}{1-x})$. Since $\varphi (x)=1-x$ belongs to $L^{\infty}(\frac{1}{1-x})$ and $T$ preserves lattice structure, it 
sufficies to see that $T\varphi\in L^{\infty}(\frac{1}{1-x})$. We have 
$$
T\varphi(x) = \varphi (\sigma(x)) = 1-\sigma(x) = 1 - \frac{x}{x + e - ex} = \frac{(1-x)e}{x+e-ex}
$$
and so
$$
\frac{T\varphi(x)}{1-x} = \frac{e}{x+e-ex}\leq e ~~ {\rm for ~ each} ~~ x \in [0, 1],
$$
which means that 
$$
\|T\|_{L^{\infty}(\frac{1}{1-x})\rightarrow L^{\infty}(\frac{1}{1-x})}\leq e.
$$
On the other hand, for $h\in L^1(\frac{1}{1-x})$
$$
\|Th\|_{L^{1}(\frac{1}{1-x})} = \int_0^1 \frac{h(\sigma(x))}{1-x} dx
$$
and changing variables $\sigma(x)=u$, $dx =\frac{e}{(1-u+eu)^2}du$ we obtain
$$
\|Th\|_{L^{1}(\frac{1}{1-x})} = \int_0^1 \frac{h(u)}{1-u}\frac{e}{1- u+eu}du\leq e\int_0^1 \frac{h(u)}{1-u}du.
$$
Thus, once again 
$$
\|T\|_{L^{1}(\frac{1}{1-x})\rightarrow L^1(\frac{1}{1-x})}\leq e.
$$
To finish the proof notice that $T$ is a bijection and so, in particular, for each $g\in S$, there is $h\in S$ such that $h_d=g$. 
\endproof

%%%%%%%%% Remark 12
\begin{remark}
The proof of Theorem \ref{thm:dualgeneralon01a} is true also for Banach ideal spaces $X$ with the Fatou property in which the above composition operator $T$ is bounded, for example, in weighted $L^p(w)$ spaces when the weight $w$ is increasing or power function 
on $I = (0, 1]$. Lemma \ref{AM09} and so Theorem \ref{thm:dualgeneralon01b} are true in some non-symmetric spaces like weighted 
$L^p$-spaces (see Appendix B). In particular, Theorems \ref{thm:dualgeneralon01b} and \ref{thm:dualgeneralon01a} together with this remark gives that
$$
\left( Ces_{p, x^{\alpha}}\right)^{\prime} = \left[ C\big(L^p(x^{\alpha})\big) \right]^{\prime} 
= \widetilde{\big(L^{p^{\prime}}\big(\frac{1}{x^{\alpha}(1-x)}\big) \big)},
$$
provided $ \alpha < 1-1/p$ and $1 \leq p < \infty$.
\end{remark}

%%%%%%%%%%% Corollary 13
\begin{corollary}
If $X$ is a symmetric space on $[0, 1]$ with the Fatou property such that $C, C^*: X \rightarrow X$ are bounded, then
$$
(CX)^{\prime} = \widetilde{ X^{\prime}\big(\frac{1}{1-x}\big)}.
$$
\end{corollary}

%%%%%%%%%%%%%%%%%%%%%%%%%%%%%%%%%%% Section 5
\section{ Duality in the sequence case}

Using analogous method we can prove the duality theorem also for Ces\`aro sequence spaces. Only in this section $C$ will stand for 
the {\it discrete Ces\`aro operator $C$}, which is defined on a sequence $x = (x_n)$ of real numbers by
$$
(Cx)_n=\frac{1}{n}\sum_{k=1}^n x_k,\ n\in \mathbb{N}.
$$
We define also the nonincreasing majorant $\widetilde{x}$ of a given sequence $x$ by
$$
(\widetilde{x})_n=\sup_{k \in {\mathbb N}, \, k \geq n} |x_k|,\ n\in \mathbb{N}.
$$
If $X$ is a Banach ideal sequence space, we define an {\it abstract Ces\`aro sequence space} $CX$ as 
$$
CX=\{x\in \mathbb{R}^{\mathbb{N}}: C|x|\in X\} ~~ {\rm with ~the ~norm } ~~ \|x\|_{CX}=\|C|x|\|_{X}.
$$
The space $\widetilde{X}$ is defined as before with evident modification. 
It is not difficult to see that $CX\not = \{0\}$ if and only if $(\frac{1}{n})\in X$. 
Note that if $(\frac{1}{n})\in X_a$, then $\| e_k\|_{CX} = \| C(e_k)\|_X = \| (\frac{1}{n}) \chi_{[k, \infty)}\|_X \rightarrow 0$ as 
$k \rightarrow \infty$, which means that $CX$ is not a symmetric space.

If $x = (x_n)$ and $m \in \mathbb N$, then the dilations $\sigma_m x$ are defined by (cf. \cite[p. 131]{LT79} and \cite[p. 165]{KPS82}):
$$
\sigma_m x = \left( ( \sigma_m x)_n \right)_{n=1}^{\infty} = \big (x_{[\frac{m-1+n}{m}]} \big)_{n=1}^{\infty} 
= \big ( \overbrace {x_1, x_1, \ldots, x_1}^{m}, \overbrace {x_2, x_2, \ldots, x_2}^{m}, \ldots \big).
$$
We have the following duality result.

%%%%%%%%%%%%%%%%%%%%% Theorem 6
\begin{theorem}\label{thm:dualsequence}
Let $X$ be an ideal Banach sequence space such that the Ces\`aro operator $C$ is bounded on $X$ and the dilation 
operator $\sigma_{3}$ is bounded on $X^{\prime}$. Then 
\begin{equation} \label{Thms}
(CX)^{\prime} = \widetilde{X^{\prime}}  ~~~~ {\it with ~equivalent ~norms.}
\end{equation}
\end{theorem}

%%%%%%%%%%%% Proof
\proof
For the inclusion $(CX)^{\prime} \hookrightarrow \widetilde{X^{\prime}}$ it is enough to follow the proof of 
Theorem \ref{thm:dualgeneralR}. Let us only notice that the corresponding result of Sinnamon, which is the main contribution there, 
holds also for counting measure on $\mathbb N$. As before, in this part of the proof we need $C$ to be bounded on $X$. 

We will comment the inclusion $(CX)^{\prime} \hookleftarrow \widetilde{X^{\prime}}$ a little more careful just because one cannot 
make the respective substitution in the case of series in place of integral. First of all we can reformulate just slightly the inequality 
from \cite{CR13}. Proving exactly like there one has that for $0\leq x \in \mathbb{R}^{\mathbb{N}}$
$$
\sum_{k=1}^n\frac{1}{k}\sum_{j=1}^kx_j\geq \sum_{j=1}^{[\frac{n+1}{2}]}x_j\sum_{k=j}^n\frac{1}{k}\geq
\frac{1}{4}\sum_{j=1}^{[\frac{n+1}{2}]}x_j,
$$
just because $\sum_{k=[\frac{n+1}{2}]}^n\frac{1}{k}\geq \frac{1}{4}$ for $n \in \mathbb{N}$. Then  
\begin{equation} \label{majorseq}
\sum_{j=1}^{n}x_{[\frac{j+2}{3}]} \leq 
3\sum_{j=1}^{[\frac{n+1}{2}]}x_j\leq 12 \sum_{j=1}^n (Cx)_j,
\end{equation}
where the details of the first estimation are left for the reader. Finally, for $y\in \widetilde{X'}$ and $0\leq x \in CX$, we 
put $\widetilde{y}=b$ and by (\ref{majorseq}) and H\"older-Rogers inequality
\begin{eqnarray*}
\sum_{n=1}^{\infty}|y_nx_n|
&\leq& 
\sum_{n=1}^{\infty}\widetilde{y}_nx_n= \frac{1}{3} \sum_{n=1}^{\infty} b_{[\frac{n+2}{3}]}x_{[\frac{n+2}{3}]} \leq 
4 \sum_{n=1}^{\infty} b_{[\frac{n+2}{3}]}(Cx)_n \\
&\leq& 4 \, \| \sigma_{3}b\|_{X^{\prime}} \| Cx\|_X \leq 4\, \| \sigma_{3}\|_{X'\rightarrow X'} \|b\|_{X^{\prime}} \| Cx\|_X \\
&=& 
4 \, \| \sigma_{3}\|_{X^{\prime}\rightarrow X^{\prime}} \|y\|_{\widetilde{X^{\prime}}} \| x\|_{CX}. 
\end{eqnarray*}
Thus $\widetilde{X^{\prime}}  \overset{D}{\hookrightarrow } (CX)^{\prime}$ with 
$D = 4 \, \| \sigma_{3}\|_{X^{\prime}\rightarrow X^{\prime}}$.
\endproof

Particular duality results for Ces\`aro sequence spaces $ces_p$  and weighted Ces\`aro sequence spaces $ces_{p, w}$
were proved by several authors. %All these duality results follows from Theorem \ref{thm:dualsequence}. 
Already in 1957 
Alexiewicz \cite{Al57} showed that for weight $w = (w_n)$ with $w_n \geq 0, w_1 > 0$ we have 
$\left(\widetilde{l^1(w)} \right)^{\prime} \equiv ces_{\infty, v}$, where $v(n) = \frac{n}{\sum_{k=1}^n w_k}$. Using the Fatou 
property of the space $\widetilde{l^1(w)}$ we obtain
\begin{equation} \label{Alex}
\left( ces_{\infty, v} \right)^{\prime} \equiv \left(\widetilde{l^1(w)} \right)^{\prime \prime} \equiv \widetilde{l^1(w)}.
\end{equation}

Jagers \cite{Ja74} presented isometric description of $\left( ces_{p, w} \right)^{\prime}$ for $1 < p < \infty$, which is not so easy to 
present shortly. Ng and Lee \cite{NL76} extended Jagers result on duality $(ces_{\infty, w})^{\prime}$ under additional assumption 
that $w_n \geq w_{n+1}$ for all $n \in \mathbb N$. Let us notice that the result in (\ref{Alex}) is simpler and precisely described.

Bennett \cite{Be96}, using factorization technique (technique where we replace the classical inequalities by identities), showed 
that for $1 < p < \infty$ we have $(ces_p)^{\prime} = \widetilde{(l^{p^{\prime}})}$ with equivalent norms. This identification follows 
from his Theorems 4.5 and 12.3 in \cite{Be96}. Moreover, on page 62, he proved (\ref{Alex}) with the equality of norms.

Grosse-Erdmann \cite[Corollary 7.5]{GE98}, using the blocking technique, was able to show Bennett's result on duality 
$(ces_p)^{\prime} = \widetilde{(l^{p^{\prime}})}$ with equivalent norms (for $1 < p < \infty$). He also has some weighted 
generalizations and duality results for more general sequence spaces. Blocking technique allow to replace sequence space 
which is quasi-normed by an expression in the section form by equivalent quasi-norm in the block form and vice versa:
$$
(\sum_{n=1}^{\infty} [a_n (\sum_{k=1}^n |x_k|^p)^{1/p}]^q)^{1/q} \approx 
(\sum_{\nu = 0}^{\infty} [2^{-\nu \alpha} (\sum_{k\in I_{\nu}} |x_k|^p)^{1/p}]^q)^{1/q},
$$
where $\{I_{\nu}\}$  is a partition of the natural numbers into disjoint intervals. Often, but not always, $I_{\nu}$  may be taken as 
the dyadic block $[2^{\nu}, 2^{\nu+1})$. The disadvantages of blocking technique is loosing the control over constants, it 
does not convey the best-possible constants and can be used only for $l^p$ or weighted $l^p$ spaces, not for more general spaces.

For example, if $1 \leq p < \infty$ and $\alpha < 1 - 1/p$, then the discrete Ces\`aro operator $C$ is bounded in weighted spaces 
$l^p(n^{\alpha})$ (see Hardy-Littlewood \cite{HL27} and Leindler \cite{Le70} with more general weights and $\| C \|_{l^p(n^{\alpha}) \rightarrow l^p(n^{\alpha})} \leq p \frac{(1 - \alpha) p}{(1 - \alpha) p - 1}$). Moreover, we can easily prove that 
$\| \sigma_{3} \|_{l^p(n^{\alpha}) \rightarrow l^p(n^{\alpha})} \leq 3^{1/p} \max(1, 3^{\alpha})$ and from 
Theorem \ref{thm:dualsequence} we obtain duality 
\begin{equation} \label{weight}
(ces_{p, \alpha})^{\prime} = (Cl^p(n^{\alpha}))^{\prime} =  \widetilde{l^{p^{\prime}}(n^{-\alpha})}
\end{equation}
with equivalent norms. This result was also proved in \cite[Theorem 7.2]{GE98} by use of the blocking technique. Our method gives these results as well, but is much simpler. Moreover, our Theorem \ref{thm:dualsequence} covers for example  Ces\`aro-Orlicz sequence spaces (cf. \cite{MPS07}) or  
weighted Ces\`aro-Orlicz sequence spaces, where the blocking technique, Jagers' method or Bennett's factorization seem to be not applicable.

%%%%%%%%%%%%%%%%%%%%%%%%%%%% Section 6
\section{\protect Extreme case and applications}

First, we give a simple proof of a generalization of the Luxemburg-Zaanen \cite{LZ66} and Tandori \cite{Ta55} duality result to weighted 
$L^{\infty}$-spaces. 
They proved that $(Ces_{\infty}[0, 1])^{\prime} \equiv (CL^{\infty}[0, 1])^{\prime} \equiv \widetilde{L^1[0, 1]}$.

%%%%%%%%%%%%%%%% Theorem 7
\begin{theorem} \label{thm:LZ}
Let either $I = [0, 1]$ or $I = [0,\infty)$. If a weight $w$ on $I$ is such that $W(x) = \int_0^x w(t) dt < \infty$ for any $x \in I$ and 
$v(x) = \frac{x}{W(x)}$, then
\begin{equation} \label{Thm6}
(Ces_{\infty, v})^{\prime} \equiv (CL^{\infty}(v))^{\prime} \equiv  \widetilde{L^1(w)}.
\end{equation}
\end{theorem}

%%%%%%%%% Proof
\proof
Let $g \in  \widetilde{L^1(w)}$ and $f \in X:= CL^{\infty}(v)$ with the norm $\| f \|_X = 1$. Then
$$
\int_0^u |f(x)| dx \leq W(u) \| f\|_X = \int_0^u w(x) dx ~~ {\rm for ~all} ~~ u \in I
$$
and since $ \widetilde{g}$ is decreasing on $I$ we obtain
$$
\int_I |f(x)| \, \widetilde{g}(x) dx \leq \int_I w(x) \, \widetilde{g}(x) dx,
$$
and so
$$
\int_I |f(x) g(x)| dx \leq \int_I |f(x)| \, \widetilde{g}(x) dx \leq \int_I w(x) \, \widetilde{g}(x) dx.
$$
Thus $\| g \|_{X^{\prime}} \leq \| g \|_{ \widetilde{L^1(w)}}$ and $ \widetilde{L^1(w)} \overset{1}{\hookrightarrow } X^{\prime}$.

On the other hand, for $g_{W,  t}(x) = \frac{1}{W(t)} \chi_{[0, t]}(x)$ with $t, x \in I$ and $f \in  \left(\widetilde{L^1(w)} \right)^{\prime}$ we have
$$
\| g_{W, t} \|_{\widetilde{L^1(w)}} = \| \widetilde{g_{W, t}} \|_{L^1(w)} =  \| g_{W, t} \|_{L^1(w)} = \frac{1}{W(t)} \|  \chi_{[0, t]} \|_{L^1(w)} 
= \frac{\int_0^t w(x) dx}{W(t)} = 1
$$
and
\begin{eqnarray*}
\| f \|_{ \left(\widetilde{L^1(w)} \right)^{\prime}}
&=&
\sup_{\| g \|_{\widetilde{L^1(w)}} = \,1} \int_I |f(x) g(x)| dx \geq \sup_{t \in I}  \int_I |f(x) g_{W, t} (x)| dx \\
&=& 
\sup_{t \in I} \frac{ \int_0^t |f(x)| dx} {W(x)} = \sup_{t \in I} C|f|(t) \frac{ t} {W(t)} = \| C|f| \|_{L^{\infty}(v)} = \| f \|_X.
\end{eqnarray*}
This means $\left (\widetilde{L^1(w)} \right)^{\prime} \overset{1}{\hookrightarrow } X$ or 
$$
 X^{\prime} \overset{1}{\hookrightarrow } \left( \widetilde{L^1(w)} \right)^{\prime \prime} \equiv  \widetilde{L^1(w)}.
 $$
 \endproof

The above proof works as well in the case of sequence spaces and gives (\ref{Alex}).
%\vspace{2mm}

Section 4 of the paper \cite{DS07} was devoted to the identification of Ces\`aro spaces $CX$ with $X$ being 
the Lorentz space $\Lambda_{\varphi}$ defined on $I = [0, \infty)$ by
$$
\Lambda_{\varphi}=\{f\in L^0: \| f \|_{\Lambda _{\varphi}}=\int_{I}f^{\ast}(t) d\varphi(t) < \infty \},
$$
where $\varphi$ is a concave, positive and increasing function on $I$ with  $\varphi(0)=0$. We shall demonstrate, that the result proved 
in \cite[Theorem 4.4]{DS07} is a straightforward consequence of our duality result in Theorem 3. 

%%%%%%%%%%%%%%%% Theorem 8
\begin{theorem}  \label{thm:DS}
For a Lorentz space $\Lambda_{\varphi}$ on $I = [0,\infty)$ with $\varphi$ satisfying $\varphi(0^+) = 0$ and for which there 
are constants $c_1, c_2 > 0$ such that 
\begin{equation} \label{Thm7}
\int_0^t \frac{\varphi(s)}{s} \, ds \leq c_1 \varphi(t), ~~\int_t^{\infty}  \frac{\varphi(s)}{s^2} \, ds \leq c_2  \frac{\varphi(t)}{t} ~~{\it for ~all} ~ t > 0,
\end{equation}
we have
$$
C\Lambda_{\varphi}=L^1(\varphi(t)/t).
$$
\end{theorem}

%%%%%%%%% Proof
\proof
It is known that $(\Lambda_{\varphi})^{\prime} = M_{t/\varphi(t)}$ (see \cite[Theorem 5.2, p. 112]{KPS82}), where  $M_{t/\varphi(t)}$ 
is the Marcinkiewicz space given by the norm
$$
\| f \|_{M_{t/\varphi(t)}} = \sup_{t > 0}\frac{tf^{**}(t)}{\varphi(t)} = \sup_{t > 0}\frac{\int_0^t f^{*}(s)\, ds}{\varphi(t)}.
$$
Since $\| f \|_{M^*_{t/\varphi(t)}} \leq \| f \|_{M_{t/\varphi(t)}} \leq c_1 \| f \|_{M^*_{t/\varphi(t)}}$, where 
$\| f \|_{M^*_{t/\varphi(t)}} = \sup_{t > 0}\frac{tf^{*}(t)}{\varphi(t)}$, it follows from Theorem \ref{thm:dualgeneralR} (second 
estimate in (\ref{Thm7}) ensures boundedness of operator $C$ in $\Lambda_{\varphi}$) that 
$$
(C\Lambda_{\varphi})^{\prime} = \widetilde{\Lambda_{\varphi}^{\prime}} =  \widetilde{M_{t/\varphi(t)}} =  \widetilde{M^*_{t/\varphi(t)}}.
$$
Also $\left (\widetilde{f} \right)^*=\widetilde{f}$ gives that 
\begin{eqnarray*} 
\|f\|_{\widetilde{M^*_{t/\varphi(t)}}} 
&=&
\sup_{t > 0} \frac{t\widetilde{f}(t)}{\varphi(t)} = \sup_{t > 0} {\rm ess} \sup_{s \geq t} \frac{tf(s)}{\varphi(t)} 
={\rm ess} \sup_{s > 0} \sup_{t \leq s}\frac{tf(s)}{\varphi(t)} \\
&=& 
{\rm ess} \sup_{s > 0}\frac{s |f(s)|}{\varphi(s)} = \| f \|_{L^{\infty}(t/\varphi(t))}.
\end{eqnarray*}
Therefore, $\widetilde{M^*_{t/\varphi(t)}} \equiv L^{\infty}(t/\varphi(t))$ and the result follows by the duality 
$[L^{\infty}(\varphi(t)/t)]^{\prime} \equiv L^1(t/\varphi(t))$. In fact,
$$
C\Lambda_{\varphi} \equiv (C\Lambda_{\varphi})^{\prime \prime} = (\widetilde{M_{t/\varphi(t)}})^{\prime} 
= (\widetilde{M^*_{t/\varphi(t)}})^{\prime} \equiv [L^{\infty}(t/\varphi(t))]^{\prime} \equiv L^1(\varphi(t)/t).
$$
\endproof

\vspace{-5mm}

%%%%%%%%%%%%%%%%%%%%%%%%%%%% Section 7
\section{\protect \medskip Appendix A}

We present here a simple proof of the weighted version of the Calder\'on-Mitjagin interpolation theorem. We will use notations 
from  \cite{KPS82} and \cite{BS88}. %and \cite{Cw76}.

%%%%%%%%%%%%%%%% Proposition 13
\begin{proposition}  \label{thm:CM}
Let weight $w$ and all symmetric spaces $X, L^1, L^{\infty}$ be on $I$. If $X$ is an interpolation space between $L^1$ and 
$L^{\infty}$ with $\| T \|_{X \rightarrow X} \leq C \max ( \| T \|_{L^1 \rightarrow L^1}, \| T \|_{L^{\infty} \rightarrow L^{\infty}})$, 
then $X(w)$ is an interpolation space between $L^1(w)$ and $L^{\infty}(w)$ and
\begin{equation} \label{Thm8}
\| T\|_{X(w) \rightarrow X(w)} \leq C \max ( \| T \|_{L^1(w) \rightarrow L^1(w)},  \| T \|_{L^{\infty}(w) \rightarrow L^{\infty}(w)}). 
\end{equation}
\end{proposition}

%%%%%%%%% Proof
\proof
First, note that for $f \in L^1(w) + L^{\infty}(w)$ we have
\begin{equation} \label{K}
K(t, f;  L^1(w), L^{\infty}(w)) = K(t, fw;  L^1, L^{\infty}).
\end{equation}
In fact, if $f = g + h$ is an arbitrary decomposition of $f$ with $g \in L^1(w)$ and $h \in L^{\infty}(w)$, then
$gw \in L^1, hw \in L^{\infty}$ and so
$$
K(t, fw;  L^1, L^{\infty}) \leq \| g w\|_{L^1} + t \| h w \|_{L^{\infty}} =  \| g \|_{L^1(w)} + t \| h\|_{L^{\infty}(w)},
$$
which gives that $f w \in L^1 + L^{\infty}$ or $f \in (L^1 + L^{\infty})(w)$ and 
$$
K(t, fw;  L^1, L^{\infty}) \leq K(t, f;  L^1(w), L^{\infty}(w)).
$$
On the other hand, if $f \in (L^1 + L^{\infty})(w)$ or $f w \in L^1 + L^{\infty}$, then for arbitrary decomposition $fw = g_1 + g_2$ 
with $g_1 \in L^1, g_2 \in L^{\infty}$ we take for $i = 1, 2$
$$
f_i = \frac{g_i}{w} ~~ {\rm on ~the ~support ~of} ~w ~ {\rm and} ~~ f_i = 0 ~~ {\rm elsewhere}.
$$
Then 
$$
f = \frac{g_1}{w} + \frac{g_2}{w} = f_1 + f_2 ~~ {\rm on ~the ~support ~of} ~w ~ {\rm and} ~~ f = 0 ~~ {\rm elsewhere},
$$
and so $f_1 \in L^1(w), f_2 \in L^{\infty}(w)$. Therefore,
\begin{eqnarray*}
K(t, f;  L^1(w), L^{\infty}(w)) 
&\leq& 
 \| f_1 \|_{L^1(w)} + t \| f_2\|_{L^{\infty}(w)} \\
&=& 
\| f_1 w \|_{L^1} + t \| f_2 w \|_{L^{\infty}} = \| g_1 \|_{L^1} + t \| g_2 \|_{L^{\infty}}
\end{eqnarray*}
for arbitrary decomposition $fw = g_1 + g_2$, which gives
$$
K(t, f;  L^1(w), L^{\infty}(w)) \leq K(t, fw;  L^1, L^{\infty}), 
$$
and (\ref{K}) is proved. Second, if $T: (L^1(w), L^{\infty}(w)) \rightarrow (L^1(w), L^{\infty}(w))$ is a bounded linear operator, then
$$
K(t, Tf;  L^1(w), L^{\infty}(w)) \leq \max(C_1, C_{\infty}) \,  K(t, f;  L^1(w), L^{\infty}(w))
$$
for any $fw \in L^1 + L^{\infty}$, where $C_i = \| T \|_{L^i(w) \rightarrow L^i(w)}, i = 1, \infty$. Therefore, by (\ref{K}), we obtain
$$
K(t, (Tf) \, w;  L^1, L^{\infty}) \leq \max(C_1, C_{\infty}) \,  K(t, fw;  L^1, L^{\infty}) ~~ {\rm for ~any} ~~fw \in L^1 + L^{\infty}.
$$
If now $fw \in X$, then by the Calder\'on-Mitjagin interpolation theorem (cf. \cite[Theorem 3]{Ca66}, \cite[Theorem 4.3 on p. 95]{KPS82} 
and \cite[Theorem 2.12]{BS88}) we have $(Tf) w \in X$ and
$$
\| (Tf) w\|_X \leq C  \max(C_1, C_{\infty}) \, \|f w\|_X ~~ {\rm or} ~~  \|Tf \|_{X(w)} \leq C  \max(C_1, C_{\infty}) \, \|  f\|_{X(w)}.
$$
Thus, estimate (\ref{Thm8}) is proved.
\endproof

\vspace{-5mm}

%%%%%%%%%%%%%%%%%%%%%%%%%%%% Section 8
{\section{\protect \medskip Appendix B}

We give an improvement of the Hardy inequality on $[0, 1]$.
\vspace{2mm}

%%%%%%%%%%%%%%%% Proposition 13
\begin{theorem}  \label{thm:Hardy}
If $1 \leq p < \infty$ and $\alpha < 1-1/p$, then
\begin{equation} \label{inequalityHardy}
\int_0^1 \left[ Cf(x) \, x^{\alpha} \right]^p\, dx \leq \left( C_{p, \alpha} \right)^p \int_0^1 \left[ (1-x) f(x) \, x^{\alpha} \right]^p\, dx
\end{equation}
for all $0 \leq f \in L^p((1-x) x^{\alpha})$, where $C_{p, \alpha} = \dfrac{p}{p-\alpha p - 1} \max (1, p-\alpha p - 1)^{1/p} $.
\end{theorem}

\proof
For $p = 1, \alpha < 0$ and $0 \leq f \in L^1(x^{\alpha})$ we have by the Fubini theorem
\begin{eqnarray*}
\int_0^1 \left[ x^{\alpha} Cf(x) \right] \, dx 
&=&
\int_0^1 x^{\alpha - 1} \left( \int_0^x f(t) \, dt \right) \, dx = \int_0^1 \left( \int_t^1 x^{\alpha - 1} \, dx \right) f(t) \, dt \\
&=&
\frac{1}{-\alpha} \int_0^1 (1 - t^{-\alpha}) f(t) t^{\alpha} \, dt \leq \frac{\max(1, - \alpha)}{-\alpha} \int_0^1 (1 - t) f(t) t^{\alpha} \, dt.
\end{eqnarray*}
Let $1 < p < \infty$ and $0 \leq f \in L^1[0, 1]$. Simple differentiation of $F(x) = (\int_0^x f(t)\, dt)^p$ gives equality (sometimes 
refered as Davis-Petersen's lemma -- see \cite[Lemma 2]{DP64})
\begin{equation} \label{equalityDP}
\left( \int_0^x f(t) \, dt \right)^p = p \int_0^x f(t) \left [ \int_0^t f(s) \,ds \right]^{p-1} dt.
\end{equation}
Let $0 \leq f \in L^p(x^{\alpha})$. Of course, $f \in L^1$ because
$L^p(x^{\alpha})\overset{A}{\hookrightarrow } L^1$ with $A = (1-\alpha p')^{-1/p'}$. We have
\begin{eqnarray*}
I 
&=& 
\int_0^1 \left[ x^{\alpha} \, Cf(x) \right]^p\, dx = \int_0^1 x^{(\alpha - 1) p} \left(\int_0^x f(t) \, dt \right)^p dx \\
&=& 
p \int_0^1 x^{(\alpha - 1) p} \left( \int_0^x g(t) \, t^{p-1} dt \right) dx,
\end{eqnarray*}
where $g(t) = f(t) [Cf(t)]^{p-1}$. By the Fubini theorem and the H\"older-Rogers inequality the last integral is
\begin{eqnarray*}
I 
&=& 
p \int_0^1 \left( \int_t^1 x^{(\alpha - 1) p} dx \right) g(t) \, t^{p-1} dt 
= p \int_0^1 \dfrac{1 - t^{(\alpha-1)p +1}}{(\alpha - 1)p + 1} g(t) \, t^{p-1} dt \\
&=&
\dfrac{p}{(1-\alpha)p - 1} \int_0^1 \left( t^{(\alpha -1) p + 1} -1\right) t^{-\alpha p + p - 1} g(t) \, t^{\alpha p} dt \\
&=&
\dfrac{p}{p-\alpha p - 1} \int_0^1 \left( 1- t^{p - \alpha p -1} \right) f(t) [ Cf(t)]^{p-1} \, t^{\alpha p} dt \\
&\leq&
\dfrac{p}{p-\alpha p - 1} \left( \int_0^1 \left( 1- t^{p - \alpha p -1} \right)^p f(t)^p \, t^{\alpha p} dt \right)^{1/p} \,  
\left( \int_0^1Cf(t)^p  \, t^{\alpha p} dt \right)^{1/p'}.
\end{eqnarray*}
%%%%%%%%%%%%%%%%%%%%%%%%%%%%%%%%%%%%%%%%%%%
Observe that $p-\alpha p - 1 > 0$ and so
$$
1 - t^{p - \alpha p -1} \leq \max (1, p - \alpha p -1) \, (1-t) ~~{\rm for} ~~t \in I.
$$
Really, if $p - \alpha p -1 \leq 1$, then it is clear and if $p - \alpha p -1 \geq 1$, then by the Bernoulli inequality
$$
t^{p - \alpha p -1} = (1 + t - 1)^{p - \alpha p -1} \geq 1 + (p - \alpha p -1)(t - 1),
$$
that is, $1 - t^{p - \alpha p -1} \leq (p - \alpha p -1) (1-t)$. Moreover, note that if $0 \leq f \in L^p(x^{\alpha})$ and 
$\alpha < 1 - 1/p$, then by the classical Hardy inequality (cf. \cite{HL27}, \cite{HLP52}, \cite{KMP07}) $Cf \in L^p(x^{\alpha})$.
Hence,
\begin{eqnarray*}
I 
&=& 
\int_0^1 \left[ Cf(x) \, x^{\alpha} \right]^p\, dx \\
&\leq& 
\dfrac{p}{p-\alpha p - 1} \max (1, p-\alpha p - 1)^{1/p} \left( \int_0^1 ( 1- t)^p f(t)^p \, t^{\alpha p} dt \right)^{1/p} \, I^{1/p'},
\end{eqnarray*}
and dividing by $I^{1/p'}$ we obtain
$$
\left (\int_0^1 \left[ Cf(x) x^{\alpha} \right]^p\, dx \right)^{1/p} \leq C_{p, \alpha} \left( \int_0^1 ( 1- x)^p f(x)^p \, x^{\alpha p} dx \right)^{1/p},
$$
which is (\ref{inequalityHardy}) for all $0 \leq f \in L^p(x^{\alpha})$.
Since subspace $L^p(x^{\alpha})$ is dense in $L^p((1-x) x^{\alpha})$ we can extend estimate (\ref{inequalityHardy}) to all 
$f \in L^p((1-x) x^{\alpha})$, which finishes the proof.
\endproof

%%%%%%%%%%%%%%%%%%%%%%%%%%%%%%%%%%%%%

\end{document}